\documentclass[a4paper]{article}
\usepackage{authblk}

\usepackage{hyperref,amssymb,amsmath,graphicx,verbatim,eufrak,amsthm}
\usepackage{color}
\usepackage{xcolor}
\usepackage{mdframed}

\definecolor{env_back}{gray}{0.8}
\definecolor{thm_color}{rgb}{0,0,1}

\newtheorem{thm}{Theorem}[section]
\newtheorem{cor}[thm]{Corollary}
\newtheorem{lem}[thm]{Lemma}

\newtheorem{prop}[thm]{Proposition}

\newtheorem{conj}[thm]{Conjecture}

\newtheorem*{clm*}{Claim}

\theoremstyle{definition}
\newtheorem{dfn}[thm]{Definition}
\newtheorem{exm1}[thm]{Example}

\theoremstyle{remark}
\newtheorem{rem}[thm]{Remark}

\newenvironment{lem*}[1]{\vspace{1ex}\noindent
{\bf Lemma* (#1).} [restatement]  \hspace{0.5em} \em }{ }
\newenvironment{thm*}[1]{\vspace{1ex}\noindent 
{\bf Theorem* (#1).} [restatement]  \hspace{0.5em} \em }{ }

{\begin{exm1}}%
{\hfill$\bigtriangleup\bigtriangledown\bigtriangleup$
\end{exm1} }

\newcommand{\N}{\mathbb{N}}

\newcommand{\R}{\mathbb{R}}
\newcommand{\C}{\mathbb{C}}

\newcommand{\eps}{\varepsilon}

\renewcommand{\Pr}{}
\let\Pr\relax
\DeclareMathOperator{\Pr}{\mathbb{P}}

\def\squareforqed{\hbox{\rlap{$\sqcap$}$\sqcup$}}
\def\qed{\ifmmode\squareforqed\else{\unskip\nobreak\hfil
\penalty50\hskip1em\null\nobreak\hfil\squareforqed
\parfillskip=0pt\finalhyphendemerits=0\endgraf}\fi}

\newcommand{\ignore}[1]{ }

\newcommand{\p}{\partial}

\newcommand{\F}{\mathcal{F}}
\newcommand{\G}{\mathcal{G}}
\newcommand{\B}{\mathcal{B}}

\renewcommand{\G}{G}
\renewcommand{\H}{H}
\renewcommand{\F}{\C}

\newcommand{\haar}{\mathfrak{m}}
\newcommand{\gradi}[1]{ |\nabla #1 |_{\infty} }
\newcommand{\gradmu}[1]{ |\nabla #1 |_{\mu,2} }

\newcommand{\gr}{d} 
\newcommand{\dv}{\delta} 

\date{}

\author[1]{Idan Perl \thanks{perli@post.bgu.ac.il}}
\author[2]{Maud Szusterman \thanks{maud.szusterman@gmail.com}}
\affil[1]{Department of Mathematics, Ben Gurion University of the Negev, Be'er Sheva, Israel.  }
\affil[2]{Unit\'e de Math\'ematiques Pures et Appliqu\'ees, Ecole Normale Sup\'erieure de Lyon, Lyon, France}

\title{Harmonic functions on locally compact groups of polynomial growth }

\begin{document}

\maketitle
\begin{abstract}
We extend a theorem by Kleiner, stating that on a group with polynomial growth, the space of harmonic functions of polynomial of at most $k$ is finite dimensional, to the settings of locally compact groups equipped with measures with non-compact support. 
\end{abstract}



\section{Introduction}
\label{Intro}

The space of bounded harmonic functions on locally compact groups and Riemannian manifolds has been extensively studied over the past. We refer to the papers \cite{Avez76, Aze70, KaiMa83, HebSa93} and also \cite{Ers10, Fur02} for background on this subject.

Over the last few years, there has been a growing interest in \textit{unbounded} harmonic functions. Following the lines of Colding \& Minicozzi's proof of Yau's conjecture \cite{ColMin97}, Kleiner proved the following theorem:
\begin{thm}[\cite{Kle10}]
\label{Kleiner}
Let $\G$ be a finitely generated group of polynomial growth and $S$ a symmetric generating set. Then for any $k\in\N$ the space $HF_k(\G,S)$ of harmonic functions of polynomial growth of degree at most $k$ on the Cayley graph $(\G,S)$ is finite dimensional.  
\end{thm}

Kleiner used this fact to obtain a non-trivial representation of $\G$, which he then employed to deduce a new proof of Gromov's theorem on groups with polynomial growth, namely that a finitely generated group of polynomial growth has a finite-index nilpotent subgroup. It is natural to ask whether the converse of Kleiner's theorem holds. That is, suppose we know that $\dim HF_k(\G,\mu)<\infty$ for some $k\geq 1$ and a probability measure $\mu$, is it true that $\G$ has polynomial growth? This was confirmed in  \cite{MeYa16}, for the class of finitely generated solvable groups. Their proof method suggested the consideration of a class of measures which they dubbed \textit{courteous}, and on which we will elaborate after introducing the definition. 

\begin{conj}
\label{main conjecture}
Let $\G$ be a locally compact compactly generated group, and let $\mu$ be a courteous measure. Let $HF_k(\G,\mu)$ denote the space of $\mu$-harmonic functions with polynomial growth of degree at most $k$. Then the following are equivalent:
\begin{enumerate}
\item $\G$ has polynomial growth.
\item $\dim HF_k(\G,\mu)<\infty$ for all $k\geq 1$.
\item $\dim HF_k(\G,\mu)<\infty$ for some $k\geq 1$.
\end{enumerate}
In the finitely generated case, another equivalent condition is
\begin{enumerate}
\item[4.] $\G$ has a finite-index nilpotent subgroup. 
\end{enumerate}
\end{conj}

The implication $(2)\implies (3)$ is trivial, and the implication $(4)\implies (1)$ is a standard computation and follows from the Bass-Guivarc'h formula, see \cite{Bass72}, \cite{Gui80}. As mentioned above, the implication $(3)\implies (1)$ was proved in \cite{MeYa16} for the case of solvable finitely generated groups. In a yet unpublished work \cite{PeYa18}, the implication $(3)\implies (1)$ is proved for the case of connected groups. We note that in the general locally compact case, polynomial growth does not imply the existence of a nilpotent-by-compact subgroup, see Example 7.9 in \cite{Bre07}.   

The main result of this paper is the implication $(1)\implies (2)$. Kleiner proved it for finitely generated groups equipped with a measure uniformly distributed on the generating set. The proof generalizes in a straightforward way to finitely supported measures, but difficulty arises when considering groups that are not (necessarily) finitely generated, and measures with non-compact support. Two key elements of the proof are the Poincar\'e and reverse Poincar\'e inequalities. In Section 2, we prove a modification of those, adapted to the settings of locally compact group and courteous measures.

\subsection{Acknowledgements}

The authors thank Ariel Yadin for his continual support. The first named author was partially supported by the Israel Science Foundation (grant no. 1346/15)

\subsection{Notation \& definitions}
Throughout, let $\G$ be a compactly generated locally compact group, and let $S$ be a compact generating set, i.e.\ $G=\bigcup_{n=1}^{\infty} S^n$, where $S^n=\{s_1\cdots s_n:\ s_i\in S \}$. We also assume $S$ is symmetric, in the sense that $S=S^{-1}:=\{s^{-1}:\ s\in S\}$, and that $1\in S$. Denote by $\haar$ be the left invariant Haar measure on $\G$  normalized to $\haar(S)=1$. The generating set $S$ induces a left-invariant metric on $\G$, defined by
$$
d_S(x,y):=\min \{n:\   x^{-1}y\in S^n\}.
$$
Different choices of generating sets yield metrics that are bi-Lipschitz. For an element $x\in \G$, we write $|x|=d_S(1,x)$. We say that $\G$ has \textit{polynomial growth} if there exist $c,\gr>0$ such that $\haar(S^n)\leq c n^{\gr}$ for all $n>0$. 
\begin{dfn}
A probability measure $\mu$ on $\G$ is called \textit{courteous}, if
\begin{itemize}
\item $\mu$ has a continuous density with regards to $\haar$.
\item $\mu$ is \textit{symmetric}, i.e.\ $\mu(A)=\mu(A^{-1})$ for any measurable set $A$;
\item $\mu$ is \textit{adapted}, i.e.\ the support of $\mu$ generates $\G$; 
\item $\mu$ has \textit{exponential tail}, i.e.\ $\Pr_{\mu}[|x|>t]\leq e^{-c_{\mu}t}$ for some $c_{\mu}>0$. 
\end{itemize}
\end{dfn}
An immediate example of a courteous measure in the case of finitely generated groups, is the uniform measure on a symmetric generating set $S$. The exponential tail condition is the condition that connects the metric and the measure, and it does not depend on the specific choice of generating set.

A measurable function $f:\G\to\F$ is called \textit{$\mu$-harmonic} if
\begin{align*}
f(x)=\int f(xs)d\mu(s) \quad \forall x\in\G.
\end{align*}
For a function $f:\G\to\F$ and $k\in\N$, define the (perhaps infinite) quantity
\begin{align*}
||f||_k:=\limsup_{r\to\infty} r^{-k}\cdot \sup\{|f(x)|:\ |x|\leq r\}.
\end{align*}
If $||f||_k<\infty$, we say that $f$ has \textit{polynomial growth of degree at most $k$}. Note that $||f||_k<\infty$ is equivalent to $|f(x)|\leq c(1+|x|)^k$ for some $c>0$ and all $x\in \G$. 

We are now ready to define $HF_k(\G,\mu)$, which is the main object of interest in this work. Let
\begin{align*}
HF_k(\G,\mu):=\{f:\G\to\F:\ ||f||_k<\infty,\ \ \text{$f$ is $\mu$-harmonic} \}. 
\end{align*}
The space $HF_k(\G,\mu)$ is the space of $\mu$-harmonic functions of polynomial growth of degree at most $k$. By \cite[Proposition I.6]{Aze70}, any $\mu$-harmonic function is continuous. The group $\G$ acts on functions on $\G$ by left translations and we note that $||g.f||_k=||f||_k$. Moreover, since the group acts from the left and harmonicity is checked on the right, the space $HF_k(\G,\mu)$ is a $\G$-invariant subspace of $\F^G$. Also, the space $HF_k(\G,\mu)$ does not depend on choice of generating set. We note however, that it highly  depends on the measure $\mu$. 

\subsubsection*{Courteous measures.} We briefly discuss the motivation for considering the class of courteous measures. The nature of Conjecture \ref{main conjecture} forces one to pass freely to finite index subgroups. The following proposition gives the motivation for considering the class of courteous measures. Let $\G$ be a compactly generated locally compact group equipped with a measure $\mu$, and let $(X_t)_t$ be a discrete time random walk such that the increments $X_t^{-1}X_{t+1}$ are i.i.d.\ $\mu$. Let $\H$ be a subgroup of $\G$. Let
\begin{align*}
\tau_{\H} := \inf \{t\geq 1 :\ X_t\in \H\}
\end{align*}
be the \textit{return time} of $\H$, and let $\mu_H$ be the law of $X_{\tau_H}$.
When $\mu$ is a generating measure and $\H$ is a finite index subgroup, it is well known that $\tau_{\H}$ is a.s.\ finite, hence $\mu_{\H}$ is well defined. The following fact is the principal idea behind considering the class of courteous measures.

\begin{prop}[\cite{BE95} Lemma 3.4, \cite{MeYa16} Proposition 3.4] Let $\G$ be a compactly generated locally compact group, $\mu$ a courteous measure, and $\H$ a finite index subgroup. Then $\mu_{\H}$ is a courteous measure on $\H$, and the restriction map $f\mapsto f|_{\H}$ is a linear bijection from $HF_k(\G,\mu)$ to $HF_k(\H,\mu_{\H})$.
\end{prop}
Put simply, by passing to a finite index subgroup, the space of harmonic function of polynomial growth of at most $k$ is essentially the same. This proposition is the motivation for working in the class of courteous measures, and not just compactly supported ones.

\subsection{Statement of main result \& corollaries}
The main result of this paper is implication $(1)\implies (2)$ of Conjecture \ref{main conjecture}. We prove:

\begin{thm}
\label{main theorem}
Let $\G$ be a compactly generated locally compact group of polynomial growth, $\mu$ a courteous measure, and $k\geq 1$. Then the space $HF_k(\G,\mu)$ of $\mu$-harmonic functions with polynomial growth of degree of at most $k$ is finite dimensional.
\end{thm} 

In a related work \cite{MPTY17}, a structure theorem for the space $HF_k(\G,\mu)$ is given, under the assumption that $\mu$ is courteous and this space is finite dimensional. To state this result, we need to define the notion of a \textit{polynomial} on a group. 
\begin{dfn}[Polynomial]
\label{def:poly}
Given $f:\G\to\F$ and an element $u\in \G$ we define the \emph{left derivative} $\p_uf$ of $f$ with respect to $u$ by
$\p_uf(x)=f(u x)-f(x)$; that is, $\p_u f = u^{-1}f-f$, where $u^{-1}f$ is the left action of $u^{-1}$ on $f$.

Let $\H < \G$ be a subgroup. A function  $f:\G \to \F$ is called a \emph{polynomial with respect to $H$} if there exists some integer $k\geq 0$ such that
\[ \p_{u_1}\cdots\p_{u_{k+1}}f = 0\text{ for all }u_1,\ldots,u_{k+1}\in H \]
The \emph{degree} (with respect to $H$) of a non-zero polynomial $f$ is the smallest such $k$. When $\H=\G$ we simply say that  $f:\G\to\F$ is a \emph{polynomial}.
We denote the space of polynomials on $G$ $P^k(G)$. 
\end{dfn}

The structure theorem states:

\begin{thm}[\cite{MPTY17}]
\label{thm: polys}
Let $\G$ be a finitely generated group, $\mu$ a courteous measure, and $k\geq 1$. Suppose $\dim HF_k(\G,\mu)<\infty$. Then there is a finite-index normal subgroup $\H$ of $\G$ such that any $f\in HF_k(\G,\mu)$ is a polynomial of degree at most $k$ with respect to $H$. 
\end{thm}

It is then deduced, using a result about the Laplace operator and the fact that $\dim P^k(\G)$ is independent from $\mu$, that $\dim HF_k(\G,\mu)$ is independent from $\mu$. In conjunction with Theorem \ref{main theorem}, this gives the following.
\begin{cor}
\label{cor:dim.indep.of.mu}
Let $\G$ be a finitely generated group with polynomial growth. Then $\dim HF_k(\G,\mu)$ is finite and independent of $\mu$ for any $k\geq 1$ and courteous $\mu$.
\end{cor}

For non-discrete groups, Theorem \ref{thm: polys} does not hold. Let us briefly present a counter example. Consider the group (appearing also in \cite[Example 7.9]{Bre07}) $G=\R\ltimes (\R^2\times \R^2)$ where $\R$ acts on $(\R^2\times \R^2)$ by a dense one-parameter subgroup of $(SO(2,\R)\times SO(2,\R))$. This group is connected and has polynomial growth, but is not nilpotent-by-compact. The following is a courteous probability measure on $G$: with probability $1/2$, choose an element $(a,(0,0))$ with $a\sim U[-1,1]$, and with probability $1/2$, choose an element $(0,(u,v))$ with $u$ and $v$ are i.i.d.\ on the unit disc in $\R^2$. It is straight forward to verify that the function $f:G\to \C$ defined by $(a,(u,v))\mapsto u$ (where $u$ is seen as a complex number) is harmonic, but is not a polynomial (of any degree). 

However, if $G$ is connected and nilpotent, the harmonic functions are in fact polynomials. 

\begin{thm}
\label{thm: conn. polys}
Let $G$ be a connected locally compact and compactly generated group, and let $\mu$ be a courteous measure. If $G$ is nilpotent then for all $k\geq 1$, $HF_k(G,\mu)\subset P^k(G)$.
\end{thm}

We prove Theorem \ref{thm: conn. polys} in section \ref{section: polys}. We conjecture the following.

\begin{conj}
Let $G$ be a connected locally compact and compactly generated group, and let $\mu$ be a courteous measure. Then $G$ has a finite index nilpotent subgroup if and only if $\dim HF_k(G,\mu)<\infty$ and there exists a finite index subgroup $H$ of $G$ such that any $f\in HF_k(G,\mu)$ is a polynomial of degree at most $k$ with respect to $H$. 
\end{conj}

In the next section, we prove two inequalities which are key in the proof of Theorem \ref{main theorem}, namely Poincar\'e and reverse Poincar\'e inequality. In the last section we prove the Theorem \ref{main theorem}.

\section{Poincar\'e and reverse Poincar\'e inequalities}
\label{section: Poincares}
For any measurable set $B$, let $|B|:=\haar(B)$. Also, for $R>0$, let $B(x,R)=\{y\in G:\ |x^{-1}y| \leq R\}$, and for $a>0$ let $a B(x,R)=B(x,aR)$. Throughout this section, assume $\G$ is a compactly generated locally compact group, $S$ is a compact symmetric generating set, and $\mu$ is a courteous measure.

\subsection{Poincar\'e inequality}

Define the following notion of a gradient on functions $f:\G\to\F$:
\begin{align*}
\gradi{f}(x):=\sup \{|f(xs)-f(x)|:\ s\in S\}.
\end{align*}
The following is a Poincar\'e inequality with regards to this gradient. Subsequently, we will modify it to get a version that better suits our goal.

\begin{lem}[\cite{HMT17}, Corollary 8.5]
\label{lem: infinity Poincare}
Let $B=B(x_0,R)$ for some $x_0\in\G$ and $R\geq 1$. Let $f_B:=\frac{1}{|B|}\int_B fd\haar$. Suppose $\G$ has polynomial growth. Then
\begin{align*}
\int_B |f-f_B|^2d\haar \leq (2R)^2 \frac{|2B|}{|B|} \int_{3B} \gradi{f}(x)^2 \haar(x).
\end{align*}
\end{lem}

The following notion of a gradient is the one we will use throughout the proof of Theorem \ref{main theorem}:
\begin{align*}
\gradmu{f}(x):=\sqrt{\int_{\G} |f(xs)-f(x)|^2 d\mu(s) }.
\end{align*}
For a set $B$ with positive measure, define also:
\begin{align*}
|\nabla f|_{B,1}(x):= \frac{1}{|B|}\int_{B} |f(xs)-f(x)| d\haar(s).
\end{align*}
The following lemma is straight forward and we omit the proof. 
\begin{lem}
\label{lem: gradients}
If $\frac{d\mu}{d\haar}\geq c>0$ on $S^n$, then $|\nabla f|_{S^n,1}(x)\leq c^{-1/2} \gradmu{f}(x)$.
\end{lem}

For a compact set $K$, define the seminorm 
$$
||f||_K:=\int_K f(x)^2d\haar(x).
$$

The following proposition plays a key role in the proof of Theorem \ref{main theorem}, and might also be of independent interest. The proof is based on \cite{Tes08}.

\begin{prop}[Poincar\'e inequality]
\label{prop: courteous Poincare inequality}
Suppose $\G$ has polynomial growth, and suppose $\frac{d\mu}{d\haar}\geq c>0$ on $S^2$. Then
\begin{align*}
\int_B |f(x)-f_B|^2d\haar(x) \leq c^{-1}32R^2 \frac{|2B|^2}{|B|^2} \int_{3B} \gradmu{f}(x)^2 \haar(x).
\end{align*}
\end{prop}
\begin{proof}
Let $Pf(x):=\frac{1}{|B|}\int_{B(x,1)} f(s)d\haar(s)$. Let $y\in B(x,1)$. We have
\begin{align*}
|Pf(x)-Pf(y)| &\leq |Pf(x)-f(x)| + |Pf(y)-f(x)|	
\\  &\leq 
\frac{1}{|B|}\int_{B(x,1)} |f(s)-f(x)|d\haar(s) + \frac{1}{|B|}\int_{B(y,1)} |f(s)-f(x)|d\haar(s)
\\ &\leq 
\frac{|2B|}{|B|}\frac{2}{|2B|}\int_{B(x,2)} |f(s)-f(x)|d\haar(s)
\\  &=
2\frac{|2B|}{|B|} |\nabla f|_{S^2,1}(x).
\end{align*}
Hence, 
\begin{align}
\label{3}
\gradi{Pf}(x)\leq 2\frac{|2B|}{|B|} |\nabla f|_{S^2,1}(x).  
\end{align}
By the triangle inequality in the form $b-|b-a| \leq a$, and by applying Lemma \ref{lem: infinity Poincare} on the function $Pf$, we have
\begin{align*}
\Big( ||f||_B - \Big| ||f||_B - ||Pf||_B\Big| \Big)^2
\leq ||Pf||_B^2
\leq (2R)^2 \frac{|2B|}{|B|} \cdot ||\gradi{Pf}||^2_{3B},
\end{align*}
and by the reverse triangle inequality,
\begin{align*}
\Big| ||f||_B - ||Pf||_B \Big|
\leq ||Pf-f||_B \leq || |\nabla(f)|_{S,1}||_B.
\end{align*}
Inserting this back, we get 
\begin{align*}
\Big( ||f||_B - || |\nabla(f)|_{S,1} ||_B \Big)^2 \leq
(2R)^2 \frac{|2B|}{|B|} \cdot ||\gradi{Pf}||^2_{3B}.
\end{align*}
If $||f||_B  \leq 2 || |\nabla(f)|_{S,1}||_B$, then by Lemma \ref{lem: gradients} we are done. Otherwise, by \eqref{3},
\begin{align*}
||f||_B^2 \leq 32R^2 \cdot \frac{|2B|^2}{|B|^2} \cdot |||\nabla f|_{S^2,1}||^2_{3B}
\end{align*}
and again by Lemma \ref{lem: gradients} we are done.
\end{proof}

\begin{rem}
The assumption that $\frac{d\mu}{d\haar}\geq c>0$ on $S^2$ can be dropped. Indeed, by replacing $\mu$ by some convolutional power $\mu^{*n}$, we can ensure $\frac{d\mu^{*n}}{d\haar}\geq c>0$ on $S^2$. Moreover, $HF_k(G,\mu)$ embeds canonically in $HF_k(G,\mu^{*n})$, so in Theorem \ref{main theorem} it is enough to show that the latter is finite dimensional.
\end{rem}

\begin{rem}
In the finitely generated case, the proof of Proposition \ref{prop: courteous Poincare inequality} is significantly simpler, and is a slight modification of the Poincar\'e inequality in \cite{Kle10}, attributed to Saloff-Coste. To the best of our knowledge, the above adaptation to general locally compact groups does not appear in literature. 
\end{rem}
 
%
%

\subsection{Reverse Poincar\'e inequality}

Before proceeding to the proof of the main proposition of this section, we record a couple of useful lemmas. Recall that $|S^n|\leq cn^d$. 

\begin{lem}
\label{aux}
Let $f\in HF_k(\G,\mu)$. Suppose $|f(y)|\leq c_f(1+|y|)^k$ for some constant $c_f>0$. There exists a constant $c_2=c_2(S,\gr,k,\mu)>0$ such that the expressions
\begin{align*}
&\int_{\G\setminus B(3R)}\int_{B(2R)} |f(y)|\cdot|f(x)-f(y)|d\mu(x^{-1}y)d\haar(x) \quad \textsl{and}	\\
&\int_{B(2R)} \int_{\G\setminus B(3R)} f(y)^2 d\mu(x^{-1}y) d\haar(x)
\end{align*} 
are both bounded by $c_f^2 \cdot c_2\cdot e^{-c_{\mu} \cdot R}$ for all $R>0$. 
\end{lem}

\begin{proof}
Since the proofs are similar, we only prove the second inequality. We also assume for simplicity that $x_0=1$. We have
\begin{align*}
\int_{B(2R)} \int_{\G\setminus B(3R)} f(y)^2 d\mu(x^{-1}y) d\haar(x)
=
\int_{B(2R)} \sum_{r\geq 3R+1}\int_{|y|=r} f(y)^2 d\mu(x^{-1}y) d\haar(x).
\end{align*}
Since $|x|\leq 2R$ and $r>3R$, the exponential tail of $\mu$ implies $\int_{|y|=r} \mu(x^{-1}y)\leq e^{c_{\mu}(2R-r)}$. 
Hence,
\begin{align*}
\int_{B(2R)} \sum_{r\geq 3R+1}\int_{|y|=r} f(y)^2 d\mu(x^{-1}y) d\haar(x)
\leq 
\int_{B(2R)} \sum_{r\geq 3R+1}\int_{|y|=r} c_f^2(1+r)^2e^{c_{\mu}(2R-r)} d\haar(x),
\end{align*}
and the result follows recalling $\haar(B(2R))\leq c_S (2R)^{\gr}$. 
\end{proof}


\begin{lem}
\label{lem: unimodular}
Let $\G,\mu$ as above. Then $\G$ is unimodular, and $\rho=\frac{d\mu}{d\haar}$ is symmetric. 
\end{lem}
\begin{proof}
The first part is a specific case of \cite{HMT17}, Lemma 8.4. For the second part, note that if $\haar$ is left invariant, then the measure obtained by composing $\haar$ and the function $x\mapsto x^{-1}$ is right invariant. Since $\G$ is unimodular and since $\haar$ is normalized to $\haar(S)=1$ where $S=S^{-1}$, these two measures are equal. Thus, since $\mu$ symmetric,
$$
\int_A \rho(x^{-1})d\haar(x^{-1})=
\int_A d\mu(x^{-1})=
\int_A d\mu(x)=
\int_A \rho(x)d\haar(x)=
\int_A \rho(x)d\haar(x^{-1}))
$$
for any measurable set $A$, implying $\rho(x^{-1})=\rho(x)$.
\end{proof}


\begin{lem}
\label{lem: change of variables}
For an integrable function $f:G\times G
\to\F$,
\begin{align*}
\int_{\G}\int_{\G}f(x,y) d\mu(x^{-1}y)d\haar(x)
=
\int_{\G}\int_{\G}f(x,y) d\mu(y^{-1}x)d\haar(y)
\end{align*}
\end{lem}
\begin{proof}
We note that $\G$ is $\sigma$-finite, hence Fubini's theorem is applicable. Let $f\in L^1(\G\times \G)$. By Lemma \ref{lem: unimodular},
\begin{align*}
\int_{\G}\int_{\G}f(x,y) d\mu(x^{-1}y)d\haar(x)
& =
\int_{\G}\int_{\G}f(x,y) \rho(x^{-1}y )d\haar(y)d\haar(x)
\\ & = 
\int_{\G}\int_{\G}f(x,y) \rho(y^{-1}x )d\haar(x)d\haar(y)
\\ & = 
\int_{\G}\int_{\G}f(x,y) \rho(y^{-1}x )d\haar(y^{-1}x)d\haar(y)
\\ & = 
\int_{\G}\int_{\G}f(x,y) d\mu(y^{-1}x)d\haar(y)
\end{align*}
\end{proof}


We proceed to the main proposition.

\begin{prop}[Reverse Poincar\'e inequality]
\label{prop: courteous reverse Poincare}
Let $B=B(x_0,R)$ for some $x_0\in \G$. Let $f\in HF_k(\G,\mu)$ and suppose $|f(y)|\leq c_f (1+|y|)^k$ for some $c_f>0$ and all $y\in\G$. Then there exist constants $c_1=c_1(S,\mu)>0$ and $c_2=c_2(S,\gr,k,\mu)>0$ such that 
\begin{align*}
\int_{B} \int_G |f(x)-f(y)|^2d\mu(x^{-1}y) d\haar(y)
\leq \frac{c_1}{R^2} \int_{3B} f(x)^2 d\haar(x) + c_f^2 \cdot c_2\cdot e^{-c_{\mu} \cdot R}
\end{align*}
for all $R>0$. 
\end{prop} 

\begin{proof}
The skeleton of the argument follows the lines of \cite{ShaT10}, Lemma 7.3. To simplify notation, we will denote $f_x:=f(x)$. For convenience, we will use the following identity, obtained by the change of variables $s\mapsto x^{-1}y$. 
\begin{align*}
\int_{B(3R)} \int_{\G} | f_x- f_{xs} |^2 d\mu(s) d\haar(x)
=
\int_{B(3R)} \int_{\G} | f_x- f_y |^2 d\mu(x^{-1}y) d\haar(x)
\end{align*}
Fix $R>0$. Let $\phi$ be the cutoff function 
$$
\phi_x=
\begin{cases}
1,\ 				 & |x|\leq R	\\
\frac{2R-|x|}{R},\   & R<|x|\leq 2R 	\\
0,\ 				 & |x|>2R.
\end{cases}
$$

Since $\phi \equiv 1$ on $B(R)$, it follows that 

\begin{align}
\label{main eq}
\int_{B(R)} \int_{\G} |f_y-f_x|^2d\mu(x^{-1}y) d\haar(x)
&= 		
\int_{B(R)} \int_{\G} \phi_x^2 |f_y-f_x|^2d\mu(x^{-1}y)d\haar(x) 
\\ &\leq   
\int_{B(3R)} \int_{\G} \phi_x^2 |f_y-f_x|^2d\mu(x^{-1}y)d\haar(x).	\nonumber	
\end{align}



Note that the integrals in \eqref{main eq} are absolutely convergent since $f$ has polynomial growth and $\mu$ has exponential tail. Now, for any $x,y\in\G$ we have
\begin{align}
\label{break in three}
\phi_x^2(f_x-f_{y})
&=(f_x\phi_x^2-f_{y}\phi_{y}^2) + f_{y}(\phi_x-\phi_{y})^2
-2f_{y}\phi_x(\phi_x-\phi_{y}).	
\end{align}
Plugging \eqref{break in three} into \eqref{main eq}, we get

\begin{align*}
\int_{B(3R)} \int_{\G} \phi_x^2(f_x-f_{y})^2d\mu(x^{-1}y) d\haar(x)
&= 
\int_{B(3R)} \int_{\G} (f_x\phi_x^2-f_{y}\phi_{y}^2)(f_x-f_{y}) d\mu(x^{-1}y)	d\haar(x)
\\&
+ \int_{B(3R)} \int_{\G} f_{y}(\phi_x-\phi_{y})^2(f_x-f_{y})d\mu(x^{-1}y)	d\haar(x)
\\&
-\int_{B(3R)} \int_{\G} 2f_{y}\phi_x(\phi_x-\phi_{y})(f_x-f_{y})d\mu(x^{-1}y)d\haar(x)	\\
&:=S_1+S_2-S_3
\end{align*}

and we will bound each of the terms $S_1,S_2,S_3$ separately. 
For the first sum, we have
\begin{align*}
&S_1= 
\int_{B(3R)} \int_{\G} (f_x\phi_x^2-f_{y}\phi_{y}^2)(f_x-f_{y}) d\mu(x^{-1}y) d\haar(x)
\\&=
\int_{B(3R)} \int_{\G}  f_x\phi_x^2(f_x-f_{y}) d\mu(x^{-1}y)
-\int_{B(3R)} \int_{\G} f_{y}\phi_{y}^2(f_x-f_{y}) d\mu(x^{-1}y)d\haar(x)	
\\&=
\int_{B(3R)}  f_x\phi_x^2 \int_{\G} (f_x-f_{y}) d\mu(x^{-1}y)d\haar(x)
-\int_{B(3R)} \int_{\G} f_{y}\phi_{y}^2(f_x-f_{y}) d\mu(x^{-1}y)d\haar(x).
\end{align*}
By harmonicity of $f$, the left expression vanishes. We get:
\begin{align*}
|S_1| 
&= \Big| \int_{B(3R)} \int_{\G} f_{y}\phi_{y}^2(f_x-f_{y}) d\mu(x^{-1}y) d\haar(x) \Big|	\\
&= \Big| 
\int_{\G} \int_{\G} f_{y}\phi_{y}^2(f_x-f_{y}) d\mu(x^{-1}y) d\haar(x)
-
\int_{\G\setminus B(3R)} \int_{\G} f_{y}\phi_{y}^2(f_x-f_{y}) d\mu(x^{-1}y) d\haar(x)
\Big|.
\end{align*} 
Again by harmonicity, using Lemma \ref{lem: change of variables}, the left term vanishes. Recall that $\phi$ is supported on $B(2R)$. Hence, if $x\notin B(3R)$ and $y\in B(2R)$, the triangle inequality implies $|x^{-1}y|>R$. Therefore, 
\begin{align*}
|S_1|
&=		\Big| \int_{\G\setminus B(3R)} \int_{\G} f_{y}\phi_{y}^2(f_x-f_{y}) d\mu(x^{-1}y) d\haar(x)
 \Big|	
\\ & \leq	
\int_{\G\setminus B(3R)}\int_{B(2R)} |f_{y}|\cdot|f_x-f_{y}|d\mu(x^{-1}y)d\haar{x}
\\& \leq 
c_f^2\cdot c_2\cdot e^{-c_{\mu} R}
\end{align*}
where the last inequality is by Lemma \ref{aux}.


For the second sum, by the triangle inequality and the averages inequality $|ab|\leq \frac{1}{2}a^2+\frac{1}{2}b^2$, we get
$$
|f_{y}(f_x-f_{y})|\leq |f_{y}f_x|+|f_{y}^2|\leq \frac{1}{2}f^2_x+\frac{3}{2}f^2_{y}.
$$
Using again the fact that $\phi$ is supported on $B(2R)$, and noting that $(\phi_{y}-\phi_x)^2\leq \frac{d(y,x)^2}{R^2}$, we deduce
\begin{align*}
|S_2|
&\leq 
\frac{1}{R^2}  \int_{B(3R)} \int_{\G} (\frac{1}{2}f^2_x+\frac{3}{2}f^2_{y})d(x,y)^2 d\mu(x^{-1}y) d\haar(x) 	
\\ & = 
\frac{1}{R^2}  \int_{B(3R)} \int_{B(3R)} (\frac{1}{2}f^2_x+\frac{3}{2}f^2_{y})d(x,y)^2 d\mu(x^{-1}y) d\haar(x)
\\ & +
\frac{1}{R^2}  \int_{B(3R)} \int_{\G\setminus B(3R)} (\frac{1}{2}f^2_x+\frac{3}{2}f^2_{y})d(x,y)^2 d\mu(x^{-1}y) d\haar(x)
\\ & :=
S_{2,1}+S_{2,2}.
\end{align*}
Using Lemma \ref{lem: change of variables}, we see that the expression in $S_{2,1}$ is symmetric in $x,y$. Hence,
\begin{align*}
S_{2,1}= 2\cdot\frac{1}{R^2}\int_{B(3R)} f_x^2 \int_{B(3R)} d(x,y)^2d\mu(x^{-1}y)d\haar(x)	\leq \frac{2 \sigma^2}{R^2}\int_{B(3R)} f_x^2 d\haar(x),
\end{align*}
where $\sigma^2$ is the $\mu$-second moment of the function $x\mapsto d(1,x)$. For the other sum, using Lemma \ref{aux}, we have
\begin{align*}
S_{2,2}
&\leq 
\frac{1}{R^2}\int_{B(2R)} \int_{\G\setminus B(3R)} 
\frac{1}{2}f^2_x d\mu(x^{-1}y) d\haar(x)
\\ & +
\frac{1}{R^2} \int_{B(2R)} \int_{\G\setminus B(3R)} \frac{3}{2}f^2_{y} d\mu(x^{-1}y) d\haar(x)
\leq 
\frac{1}{R^2} \int_{B(3R)} f^2_x d\haar(x)
+ c_f^2\cdot c_2\cdot e^{-c_{\mu} R}.
\end{align*}


For the third sum, another application of the averages inequality in the form $|ab|=|\frac{1}{2}a\cdot 2b|\leq \frac{1}{4} a^2+b^2$, gives

\begin{align*}
|S_3|
&=  
\Big |\int_{B(3R)} \int_{\G}  2f_{y}\phi_x(\phi_x-\phi_{y})(f_x-f_{y})d\mu(x^{-1}y) d\haar(x)	\ \Big |	\\
&\leq  \frac{1}{2}\cdot \int_{B(3R)} \int_{\G} \phi_x^2 (f_x-f_{y})^2 d\mu(x^{-1}y) d\haar(x)
+ 2\cdot \int_{B(3R)} \int_{\G}  f_{y}^2 (\phi_x-\phi_{y})^2 d\mu(x^{-1}y) d\haar(x) \\
\end{align*} 
The left term is just half of what we wish to bound in the proposition, and the right term was already dealt with in the second sum. 


Putting the three ingredients together, we get
\begin{align*}
& \int_{B(R)} \int_{\G} (f_x-f_y)^2\mu(y^{-1}x) d\haar(x)
\leq	
\int_{B(3R)} \int_{\G} \phi_x^2(f_x-f_{y})^2d\mu(x^{-1}y) d\haar(x)
\\ & \leq 
12 c_f^2\cdot c_2\cdot e^{-c_{\mu} R}+ \frac{8 \sigma^2}{R^2}\int_{B(3R)} f_x^2 d\haar(x)
\end{align*} 
and the claim follows.
\end{proof} 

This concludes the proofs of Poincar\'e and reverse Poincar\'e inequalities. We continue to the proof of the main theorem.


\section{Proof of main Theorem}
\label{section: Kleiner}

In this section will prove our main result, Theorem \ref{main theorem}. The proof follows the lines of \cite{Kle10}, in a simplified manner that assumes doubling property: 
\begin{align}
\label{eq: Doubling property}
\exists D>0 \text{\quad s.t.} \quad \haar(B(2R)) \leq D\cdot \haar(B(R)).
\end{align}
The proof that polynomial growth implies doubling property \cite{Lo87} invokes Gromov's theorem on groups with polynomial growth \cite{Gro81}. However, since our goal here is not to prove Gromov's theorem, we may assume the doubling property. While not necessary, this significantly simplifies our proof, and helps the reader to focus on the novel parts of the proof.

Let $\mathcal{V}$ be a finite dimensional subspace of $HF_k(\G,\mu)$. We will show that the dimension of $\mathcal{V}$ is bounded by a constant that does not depend on $\mathcal{V}$, hence deducing that $\dim HF_k(\G,\mu)<\infty$. 

Denote $\dim \mathcal{V}=2\dv$. For two measurable functions $u,v:\G\to\F$, define
\begin{align*}
Q_R(u,v):=\int_{B(R)} u(x)\overline{v(x)} d\haar(x).
\end{align*}
Since $\mathcal{V}$ is finite dimensional, there exists $R_0$ such that $Q_R$ is a positive definite bilinear form for all $R\geq R_0$.

\subsubsection*{A controlled cover.} Our first step is to control the cover size of a large ball by smaller ones, and the intersection multiplicity of the covering balls. It is here that the doubling property comes into play. 

Let $\eps>0$ and $R>2\eps^{-1}$. Let $\{x_1,...,x_J\}$ be a maximal $\eps R$-separated set in $B(R)$. Let $B_j:=B(x_j,\eps R)$. The balls $\mathcal{B}=\{B_j:\ 1\leq j\leq J\}$ cover $B(R)$, and $\frac{1}{2}\mathcal{B}=\{B(x_j,\frac{\eps R}{2} ):\ 1\leq j\leq J\}$ are pairwise disjoint. 

For any measurable set $A$, let $|A|=\haar(A)$ . Since the shrunk balls are disjoint (and of same measure), we have using the doubling property
\begin{align}
\label{Cover size}
|J|\leq \frac{|B(R)|}{|B(\frac{\eps}{2}R)|}
=\frac{|B(2^{\frac{\log(2/\eps))}{\log(2)}} \cdot \frac{\eps}{2}R)|}{|B(\frac{\eps}{2}R)|}
\leq \frac{D^{\big \lceil \frac{\log(2/\eps))}{\log(2)} \big \rceil } |B(\frac{\eps}{2}R)|}{|B(\frac{\eps}{2}R)|}
=D^{\big \lceil \frac{\log(2/\eps))}{\log(2)} \big \rceil }.
\end{align}

Now, suppose $x$ is in the intersection of $\beta$ balls in $3\mathcal{B}=\{B(x_j,3\eps R ):\ 1\leq j\leq J\}$. This implies that $B(x,3.5\eps R)$ contains at least $\beta$ balls from $\frac{1}{2}\mathcal{B}$. Hence

\begin{align}
\label{intersection multiplicity}
\beta\leq \frac{|B(3.5\eps R)|}{|B(\frac{\eps}{2}R)|} 
\leq \frac{|B(2^3\cdot \frac{\eps}{2} R)|}{|B(\frac{\eps}{2}R)|}
\leq \frac{D^3|B(\frac{\eps}{2} R)|}{|B(\frac{\eps}{2}R)|}
=D^3.
\end{align}



\subsubsection*{Estimating functions relative to the cover $\mathcal{B}$.}
Our next step is to control the size of harmonic functions with regards to their averages on smaller balls. To that end, we invoke the Poincar\'e and reverse Poincar\'e inequalities. Note that unlike the compactly supported measures case, we get an error term, which we will deal with later. 

Let $\phi:\mathcal{V}\to \F^J$ be defined by $(\phi(u))_j=\frac{1}{|B_j|}\int_{\B_j}ud\haar$. Suppose $u\in \ker (\phi)$ and $\eps<\frac{1}{3}$. Assume without loss of generality that $\frac{d\mu}{d\haar}\geq c_1>0$ on $S^2$. By applying propositions \ref{prop: courteous Poincare inequality} and  \ref{prop: courteous reverse Poincare}, we have
\begin{align}
\label{Poincares}
Q_R(u,u)
&=\int_{B(R)} |u(x)|^2 d\haar(x)
\leq 
\sum_{j\in J} \int_{B_j} |u(x)|^2 d\haar(x)
=\sum_{j\in J} \int_{B_j} |u(x)- (\phi(u))_j|^2 d\haar(x)	
\\ \nonumber & \leq  
\frac{|B(2\eps R)|^2}{|B(\eps R)|^2} \cdot c_1^{-1} 32(\eps R)^2 \cdot \sum_{j\in J} \int_{3B_j} \int_{\G} |u(x)-u(xs)|^2\mu(s) d\haar(x)	
\\ \nonumber & \leq  
D^2 \cdot c_1^{-1} 32 (\eps R)^2 \cdot D^3 \cdot \int_{B(2R)}  \int_{\G} |u(x)-u(xs)|^2\mu(s) d\haar(x)	
\\  \nonumber & \leq 
C\cdot D^5 \eps^2 \cdot Q_{6R} (u,u) + 
 C\cdot  D^5 \eps^2  \cdot  R^2 \cdot 
c_u^2 \cdot e^{-c_{\mu} \cdot R}     
\end{align}
for $C=C(S,\gr,k,\mu)>0$. Note that in the second inequality we used the fact that for $\eps<\frac{1}{3}$, $3\mathcal{B}$ is contained in $B(2R)$, and the intersection multiplicity bound \eqref{intersection multiplicity}.

\subsubsection*{Controlled growth.}
In this step, we show that there are infinitely many scales $R$ for which there is a subspace $\mathcal{U}\leq \mathcal{V}$ such that the functions in $\mathcal{U}$ exhibit doubling behavior.

\begin{lem}
\label{det doubling}
There exists a constant $\Delta=\Delta(d,k)$ such that $\frac{\det(Q_{6R})}{\det(Q_R)}\leq \Delta^{\dv}$ for infinitely many $R\geq R_0$. Moreover, for any such $R$, there exists a subspace $\mathcal{U}\leq \mathcal{V}$ of dimension at least $\dv$ such that $\frac{Q_{6R}(u,u)}{Q_R(u,u)}\leq \Delta$ for any $0\neq u\in \mathcal{U}$.  
\end{lem} 
\begin{proof}
Let $R\geq R_0$. Let $B=\{u_1,...u_{2\dv}\}$ be a basis of $\mathcal{V}$ such that $|u_i(x)|\leq (1+|x|)^k$ for all $1\leq i\leq 2\dv$. Recall that $\haar(B(R))\leq c_sR^d$. We have
\begin{align*}
Q_R(u_i,u_i)=\int_{B(R)} |u_i(x)|^2 d\haar(x) \leq |B(R)|\cdot \sup \{|u_i(x)|^2:\ |x|\leq R\} \leq c_S R^d\cdot (1+R)^{2k}.
\end{align*}
Hence by Hadamard's inequality, 
\begin{align}
\label{Adam Amar}
\det(Q_R)\leq \prod_{i=1}^{2\dv}Q_R(u_i,u_i)\leq \big(c_S R^d\cdot (1+R)^{2k} \big)^{2\dv}.
\end{align}
Suppose by contradiction that $\lim_{R\to\infty} \frac{\det(6R)}{\det(R)}=\infty$. Then for any $\Delta>0$ there exists $R_\Delta$ such that $\frac{\det(6R)}{\det(R)}>\Delta^{\dv}$ for any $R\geq R_{\Delta}$. By telescoping, 
$$
\frac{\det(6^n R_{\Delta})}{\det(R_{\Delta})}>\Delta^{\dv n} \ \ \forall n>0,
$$
i.e.\ $\det(6^n R_{\Delta}) >\det(R_{\Delta})\cdot \Delta^{\dv n}$. On the other hand, by \eqref{Adam Amar} we have 
\begin{align*}
\det(6^n R_{\Delta})
&\leq \big ( c_S (6^n R_{\Delta})^d\cdot (1+(6^n R_{\Delta}))^{2k} \big)^{2\dv}	\\
&\leq c_S^{2\dv}\cdot 16^{k\dv} \cdot R_{\Delta}^{2\dv(d+2k)} \cdot  \big(6^{2(d+2k)} \big)^{\dv n},
\end{align*}
which is a contradiction for $\Delta>6^{2(d+2k)}$ and large enough $n$. 

For the second part of the claim, suppose $R$ satisfies $\frac{\det(Q_{6R})}{\det(Q_R)}\leq \Delta^{\dv}$, and let $B$ be a basis for $\mathcal{V}$ which is both $Q_R$-orthonormal and $Q_{6R}$-orthogonal. We have 
\begin{align*}
\frac{\prod_{i=1}^{2\dv} Q_{6R}(u_i,u_i)}{\prod_{i=1}^{2\dv} Q_R(u_i,u_i)}=\det(Q_{6R})\leq \Delta^{\dv},
\end{align*}
implying that there exists a subset $C\subset B$ of size at least $\dv$ such that $Q_{6R}(u,u)\leq \Delta$ for any $u\in C$. Letting $\mathcal{U}:=span\ C$, we get the desired conclusion. 
\end{proof}


\subsection*{A bound on the polynomial $k$-norm}  
For $v\in \mathcal{V}$, define
$$
c_v:=\inf \{ c:\ |v(x)|\leq c\cdot (1+|x|)^k\ \ \ \forall\ x\in \G \}.
$$

One can easily check that the function $v\mapsto c_v$ is a norm on $V$, which we dub \textit{polynomial $k$-norm}. Since $v\mapsto \big( Q_{R_0}(v,v) \big)^{1/2}$ is a norm as well, and $\mathcal{V}$ is finite dimensional, we use norm equivalency to see that 
$$
\frac{c_v^2}{Q_{R_0}(v,v)}\leq M 
$$
for some constant $M=M(\dv)>0$. Now, since $Q_R$ is increasing in $R$ , we have 
\begin{align}
\label{c_u bound}
\frac{c_v^2}{Q_R(v,v)}\leq \frac{c_v^2}{Q_{R_0}(v,v)} \leq M
\end{align}
for any $R>R_0$ and $0\neq v\in \mathcal{V}$.
%

%


\subsection*{Putting the ingredients together}
Let $\eps>0$ small enough so that $C D^5 \eps^2  <\frac{1}{2\Delta}$ (and smaller than $\frac{1}{3}$).
By Lemma \ref{det doubling} and inequality \eqref{c_u bound}, we can choose $R=R(\dv,D,S,\gr,k,\mu)>\max\{R_0,2\eps^{-1}\}$ large enough so that 
$$
C D^5 \eps^2  \cdot  R^2 \cdot 
c_u^2 \cdot e^{-c_{\mu} \cdot R} \leq  \frac{1}{4}Q_R(u,u)
$$ 
and $Q_{6R}(u,u)\leq \Delta Q_R(u,u)$ for any $u\in \mathcal{U}$, where $\dim \mathcal{U}\geq \frac{1}{2}\dim \mathcal{V}$. Plugging this into (\ref{Poincares}), we get
\begin{align*}
Q_R(u,u) & \leq \frac{1}{2\Delta} Q_{6R}(u,u) + \frac{1}{4} Q_R(u,u)
\leq \frac{3}{4} Q_R(u,u)
\end{align*}
for any $u\in \mathcal{U}\cap \ker(\phi)$, implying $Q_R(u,u)=0$, and consequently $u=0$. So $\phi:\mathcal{U}\to \F^{|J|}$ is injective, and we conclude by \eqref{Cover size} that
$$
\dim \mathcal{V}\leq 2\dim \mathcal{U} \leq 2|J|\leq 2 D^{\big \lceil \frac{\log(2/\eps))}{\log(2)} \big \rceil }.
$$
This concludes the proof of Theorem \ref{main theorem}.

\section{Harmonic functions are polynomials}
\label{section: polys}
The purpose of this section is to prove Theorem \ref{thm: conn. polys}. The proof is similar to the proof of \cite[Theorem 1.3]{MPTY17}, and we give it here for completeness.

We start by recalling some basic facts about groups acting by linear transformations. Suppose that $G$ is a group acting linearly on an $n$-dimensional vector space $V$ over a field $\C$. We denote by $\mathrm{Hom}(G,\C^\times)$ the characters of the group $G$ into the multiplicative group $\C^\times$.
Given $\lambda \in \mathrm{Hom}(G,\C^\times)$, we may denote the {\em weight space} corresponding to $\lambda$ by 
$$ V_\lambda = V_{\lambda}^{(1)} = \{ v \in V \ : \ x v = \lambda(x) v \ , \ \forall \ x \in G\} =
\bigcap_{x \in G} \ker (x-\lambda(x) I) . $$
The {\em $k$-th generalised weight space} is defined inductively by
$$ V_{\lambda}^{(k)} = \{ v \in V \ : \ (x-\lambda(x)I) v \in V_{\lambda}^{k-1} \ , \ \forall \ x \in G\} . $$
We also set $V_{\lambda}^{(0)} = \{0\}$, which is consistent with these definitions.
The {\em generalised weight space} is defined by
$$ V_{\lambda}^* = \bigcup_k V_{\lambda}^{(k)} . $$
Thus, $v \in V_{\lambda}^*$ if and only if there exists $k$ such that $(x-\lambda(x) I)^k v = 0$
for all $x \in G$.  Note that $V_{\lambda}^*$ is an $G$-invariant subspace.
It is a well-known fact from linear algebra that
$V_{\lambda}^* \cap V_{\beta}^* = \{ 0\}$ if $\lambda \neq \beta$.
It is important to note that this definition is with respect to some group acting linearly on $V$, and depends on the specific choice of the acting group.

If $G$ acts linearly on a vector space $V$ and $K$ is the kernel of this action then $G/K$ is isomorphic to a subgroup of $GL(V)$. If $G/K$ is nilpotent then we say the action of $G$ on $V$ is nilpotent. We make use of the following lemma about nilpotent linear actions. The proof is standard and employs Lie-Kolchin's theorem \cite{Kol48}. 

\begin{lem}\label{lem:eigen}
Let $G$ be a connected group, and let $V$ be a finite-dimensional vector space over $\C$ such that $G$ acts linearly on $V$ and such that this action is nilpotent. Then 
\[
V=\bigoplus_{j=1}^r V_{\lambda_j}^\ast,
\]
with $\lambda_1,\ldots,\lambda_r\in\mathrm{Hom}(G,\C^\times)$.
\end{lem}

\begin{proof}[Proof of Theorem \ref{thm: conn. polys}]
Let $k\geq 1$. By Theorem \ref{main theorem}, $\dim HF_k(G,\mu)<\infty$. The group $G$ acts linearly on $HF_k(G,\mu)$ via $g.f(x)=f(g^{-1}x)$, and since $G$ is assumed to be connected and nilpotent, we may apply Lemma \ref{lem:eigen}. Fix some $\lambda=\lambda_j$. Let $f\in V_\lambda^{(1)}$. Then for every $x\in G$ we have $f(x^{-n})=\lambda(x)^nf(1)$, which, since $f$ is bounded by a polynomial, implies that $|\lambda(x)|=1$. 
The Liouville property for nilpotent groups (see \cite{Alex87, Gui80, Kai87})  therefore implies that $f$ is constant on $G$. This implies that $V_\lambda^\ast = \{0\}$ unless $\lambda$ is the trivial character $1$, and so in fact we have $HF_k(G,\mu) = V_1^\ast$. Finally, note that $f \in V_1^{(n)}$ if and only if for all $x \in G$ we have $\p_x f = x^{-1} f-f \in V_1^{(n-1)}$. 
Since $V_1^{(0)} = \{0\}$, for every $n$ this implies that if $f:G\to\C$ belongs to $V_1^{(n)}$ then $f\in P^n(G)$. In particular, every $f\in HF_k(G,\mu)$ satisfies $f\in P^n(G)$ for $n=\dim HF_k(G,\mu)$. 
\end{proof}

\bibliographystyle{alpha}

\end{document}